\documentclass[runningheads]{llncs}
\usepackage[T1]{fontenc}
\usepackage{graphicx}
\usepackage[utf8]{inputenc}

\usepackage[appendix=append]{apxproof}
\newtheoremrep{theorem}{Theorem}
\newtheoremrep{lemma}{Lemma}
\newtheoremrep{corollary}{Corollary}
\newtheoremrep{observation}{Observation}

\usepackage{amsthm}
\usepackage{amsmath}
\usepackage{amssymb}
\usepackage{latexsym}

\usepackage{color,soul}

\usepackage{mathtools}

\usepackage{hyperref}
\usepackage{cleveref}
\usepackage{graphicx}
\usepackage{framed}
\usepackage[disableredefinitions]{complexity}

\usepackage{tikz}
\usetikzlibrary{decorations.pathreplacing,shapes}

\usepackage{framed}
\usepackage{bold-extra}
\usepackage{xspace}
 \newtheorem{rmk}{Remark}

\DeclareMathOperator{\vol}{vol}

\newcommand{\cA}{\mathcal{A}}
\newcommand{\cG}{\mathcal{G}}

\newcommand{\BB}{\mathcal{B}}

\newcommand{\ind}[1]{\mathbf{1}_{#1}}
\newcommand{\Exp}[1]{\mathbb{E}\left[#1\right]}
\newcommand{\Prob}[1]{\mathbb{P}\left(#1\right)}
\newcommand{\abs}[1]{\left\vert#1\right\vert}

\newcommand{\q}{q^*}

\newcommand{\loy}{\mathcal{L}}
\newcommand{\w}{\mathsf{w}}

\date{}
\begin{document}
\title{Approximating temporal modularity on graphs of small underlying treewidth}
%
%
\author{Vilhelm Agdur\inst{1}\orcidID{0000-0002-2278-6791} \and
Jessica Enright\inst{2}\orcidID{0000-0002-0266-3292} \and Laura Larios-Jones\inst{2}\orcidID{0000-0003-3322-0176} \and Kitty Meeks\inst{2}\orcidID{0000-0001-5299-3073} \and Fiona Skerman\inst{1}\orcidID{0000-0003-4141-7059} \and Ella Yates\inst{2}\orcidID{0009-0000-7979-0401}}
\authorrunning{V.~Agdur et al.}
\institute{Department of Mathematics, Uppsala University, Sweden \email{\{vilhelm.agdur, fiona.skerman\}@math.uu.se} \and
School of Computing Science, University of Glasgow, UK
\email{\{jessica.enright, laura.larios-jones, kitty.meeks\}@glasgow.ac.uk; 3032195Y@student.gla.ac.uk}}
\maketitle              
\begin{abstract}
Modularity is a very widely used measure of the level of clustering or community structure in networks.  
Here we consider a recent generalisation of the definition of modularity to temporal graphs, whose edge-sets change over discrete timesteps; such graphs offer a more realistic model of many real-world networks in which connections between entities (for example, between individuals in a social network) evolve over time.  
Computing modularity is notoriously difficult: it is NP-hard even to approximate in general, and only admits efficient exact algorithms in very restricted special cases. 
Our main result is that a multiplicative approximation to temporal modularity can be computed efficiently when the underlying graph has small treewidth.  This generalises a similar approximation algorithm for the static case, but requires some substantially new ideas to overcome technical challenges associated with the temporal nature of the problem.


\end{abstract}
\section{Introduction}


One of the most common methods for community detection on static graphs is \emph{modularity optimisation}, where we define an objective function on partitions of the vertices, and attempt to solve the resulting optimisation problem by various means. The computational problem is notoriously intractable: exact optimisation is known to be NP-Hard~\cite{nphard}, even restricted to $d$-regular graphs (for any $d\geq 9$) and two-part partitions~\cite{DASGUPTA201350}; moreover, it is NP-hard even to compute a multiplicative approximation \cite{dinh2015network}.  In the parameterised setting, the problem is W[1]-hard parameterised simultaneously by the pathwidth and feedback vertex number of the input graph \cite{w1hard}; the problem is however in FPT parameterised by either the vertex cover number \cite{w1hard} or max leaf number \cite{garvardtMaxLeaf24} of the input graph.  More generally, there is an fpt-algorithm parameterised by the treewidth of the input graph which computes a multiplicative approximation to the maximum modularity \cite{w1hard}.


In this paper we are concerned with computing a modularity-like quantity on networks which evolve over time.  In the real world, it is often the case that static graphs are not an appropriate model -- online social contacts change over time~\cite{Evkoski_Pelicon_Mozetic_Ljubesic_Kralj-Novak_2022}, different bird species are present in the winter and in the summer~\cite{pilosof2017multilayer}, and brain structure is affected by hormonal cycles~\cite{Mueller_Pritschet_Santander_Taylor_Grafton_Jacobs_Carlson_2021}. Temporal graphs~\cite{kempeTemporal02} are one of the most widely used models for such networks: a temporal graph is a pair $(G,\lambda)$ where $G$ is the underlying static graph and $\lambda:E(G) \rightarrow 2^{\mathbb{N}}$ assigns to each edge a (non-empty) set of timesteps at which it is active.  This model describes networks where the vertex-set remains unchanged but the edge-set changes over discrete timesteps.



The problem of community detection in temporal graphs has received increasing attention over the last twenty years, going from approaches that try to leverage community detection methods in the static case~\cite{Greene_Doyle_Cunningham_2010,Hopcroft_Khan_Kulis_Selman_2004,Palla_Barabási_Vicsek_2007,Rosvall_Bergstrom_2010,Sun_Faloutsos_Papadimitriou_Yu_2007} to online methods~\cite{Aynaud_Guillaume_2010,Folino_Pizzuti_2014,Lin_Chi_Zhu_Sundaram_Tseng_2008,Seifikar_Farzi_Barati_2020,Yuan_Zhang_Ke_Lu_Li_Liu_2025} and methods that construct an auxiliary static graph from the temporal graph and modify methods from the static case to apply to the auxiliary graph~\cite{Bassett_Porter_Wymbs_Grafton_Carlson_Mucha_2013,DiTursi_Ghosh_Bogdanov_2017,jian2023restrictedtweediestochasticblock,mucha2010community,paoletti2024benchmarkingevolutionarycommunitydetection,zhang2024fast}.

Here, we study temporal modularity as defined in Mucha et al.~\cite{mucha2010community} which generalises 
classic modularity 
to a temporal graph setting, and has become a standard method~\cite{Bianconi_2018,magnani2021analysis,Mueller_Pritschet_Santander_Taylor_Grafton_Jacobs_Carlson_2021,nicosia2013graph}.  This adaptation considers a partitioning of the vertices of the temporal graph which allows vertices to be in different parts at different times. The measure rewards partitions which have higher modularity within the static graph present at any particular time step and penalises moving vertices between parts. 
This measure exhibits a number of convenient properties: we mention several in \cref{sec:basic_facts}. Because the optimisation of this measure is a generalisation of the NP-hard static version, it is straightforwardly hard to optimise, thus motivating our interest in a parameterised approach. 




Our main result is that one of the most general positive results from the static setting can be adapted to this notion of temporal modularity.  Specifically, we show that there exists an fpt-algorithm, parameterised by treewidth of the underlying static graph (whose edge-set is the union of the edge-set at all timesteps), which computes a multiplicative approximation to the temporal modularity.  Similarly to the analogous result in the static case, this algorithm relies on a dynamic program over a tree decomposition of the underlying graph (albeit with slightly more involved accounting than is needed in the static case).  In the static case this was combined with an existing result that the maximum modularity can be approximated by considering only partitions with a constant number of parts; we prove an analogous result for the temporal case which we use here, but there is a further complication we must overcome for the temporal version.  In order to avoid the number of states that must be considered in the dynamic program over the tree decomposition becoming prohibitively large, we must restrict the number of timesteps being considered.  We show that we can in fact approximate the temporal modularity by considering separately the restriction of the temporal graph to short time-windows, provided these windows are chosen carefully.




The remainder of the paper is organised as follows.  In Section~\ref{sec:notation} we introduce notation and key definitions. In Section~\ref{sec:approx} we list the crucial lemmas that allow us to approximate temporal modularity by solving a collection of simpler problems in which both the number of permitted parts and the number of timesteps is bounded. In Section~\ref{sec:algorithm} we describe how to solve this simpler problem on graphs of small treewidth via a dynamic programming approach, and then apply the results of Section~\ref{sec:approx} to obtain our overall approximation algorithm. In Section~\ref{sec:conclusion} we conclude with some final observations and directions for future work.

\section{Notation and definitions}\label{sec:notation}
Here we introduce key definitions and notation. Throughout this work we use the Word-RAM model of computation, thus computing arithmetic operations on $\log n$-bit numbers in constant time.

A graph is a pair $G = (V,E)$ with vertex set $V$ and edge set $E$; we denote the size of the vertex set as $n = |V|$, and the size of the edge set as $m = |E|$. Graphs here are simple, loopless, and undirected.  The \emph{degree} of a vertex $v$, denoted $d_v$, is the number of other vertices it is adjacent to.  For a subset $A \subseteq V$ of the vertices of $G$, we write $E(A)$ for the set of edges between vertices in $A$, $e(A)$ for $\abs{E(A)}$, and $\vol(A) = \sum_{v \in A} d_v$. 

    Given a graph $G = (V,E)$ with $m=|E|\geq 1$ edges and a partition $\cA = \{A_1,A_2,\ldots,A_k\}$ of its vertices, the \emph{modularity} of this partition is given by
    $$q(G, \cA) = \sum_{A \in \cA} \frac{e(A)}{m} - \frac{\vol(A)^2}{(2m)^2},$$
    and the modularity of $G$ is the maximum over all partitions, that is,
    $$q^*(G) = \max_{\cA} q(G,\cA).$$
    For graphs $G'$ with no edges let $q(G',\cA)=0$ for any partition and $\q(G')=0$.

\subsection{Temporal graphs and temporal modularity}

    A \emph{temporal graph} $\cG = (G, \lambda)$ with lifetime $T$ consists of a graph $G = (V,E)$ and a function $\lambda: E \to 2^{[T]}$, assigning each edge a set of times at which it is active. For each $t \in [T]$, the \emph{snapshot} of $\cG$ at time $t$ is the graph $G_t = (V,E_t)$, where $E_t = \left\{e \in E \,\middle|\, t \in \lambda(e)\right\}$. The degree of a vertex $v$ in snapshot $G_t$ is denoted by $d_{v,t}$. We denote the number of vertices in a temporal graph as $n = |V|$, and the number of edges in a snapshot $G_t$ as $m_t = |E_t|$.

    A partition $\cA$ of a temporal graph into $k$ parts is a function $\pi_\cA: V \times [T] \to [k]$, assigning each vertex a labelled part at each time step. For each time $t$ we get a partition $\cA_t$ of $G_t$ by taking
    $$\cA_t = \left\{\pi_{\cA,t}^{-1}(1),\, \pi_{\cA,t}^{-1}(2),\,\ldots,\pi_{\cA,t}^{-1}(k)\right\},$$
    where $\pi_{\cA,t}(v) = \pi_\cA(v,t)$.  An example of such a partition is illustrated in Fig.~\ref{fig:example}.

   For a temporal graph $\cG$ and a partition $\cA$ of $\cG$, we say that a vertex $v$ is \emph{loyal} at time $t$ if $\pi_\cA(v,t) = \pi_\cA(v,t+1)$. We denote by $\loy(\cA)$ the total number of loyal vertices, over all times- we call this the \emph{loyalty contribution}. That is, if we define the \emph{indicator function} that $(v,t)$ and $(v',t')$ are in the same part by:
    $$\delta_\cA((v,t),(v',t')) = \begin{cases}
        1 &\text{if }\pi_\cA(v,t) = \pi_\cA(v',t')\\
        0 &\text{otherwise,}
    \end{cases}$$
    then we have
    $$\loy(\cA) = \sum_{v \in V} \sum_{t=1}^{T-1} \delta_\cA((v,t),(v,t+1)).$$

Adapting~\cite{mucha2010community}, specifically for temporal graphs, we can now define temporal modularity.

\begin{definition}
    Given a temporal graph $\cG = (G = (V,E),\lambda)$, a partition $\cA$ of $\cG$, and a tuning parameter $\omega \geq 0$, the \emph{temporal modularity} of $\cA$ is given by
    \begin{equation}\label{eq.defn_MRMPO}
        q_\omega(\cG, \cA) = \frac{1}{2\mu_\omega(\cG)} \left(\sum_{t=1}^T \sum_{A \in \cA_t} \left( 2e_{G_t}(A) - \frac{\vol_{G_t}(A)^2}{2m_t}\right) + \omega \loy(\cA)\right),
    \end{equation}
    where
    $$\mu_\omega(\cG) = \frac{1}{2}\omega n(T-1) + \sum_{t=1}^T m_t$$
    is a normalisation factor. The \emph{temporal modularity} of $\cG$ is given by $q^*_\omega(\cG) = \max_{\cA} q(\cG,\cA)$. 
\end{definition}
 We will often omit the dependence on $\omega$, and in most sections treat it as a given constant.

 \begin{figure}
        \centering
        \includegraphics[width=0.5\linewidth]{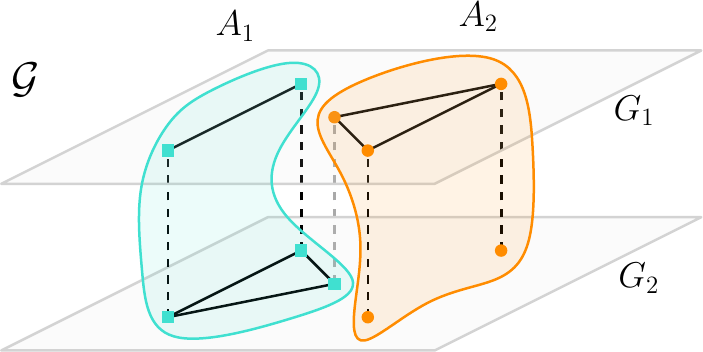}
        \caption{A temporal graph $\mathcal{G}$ together with a two part vertex partition $\cA=\{A_1, A_2\}$.  Here four of the five vertices are loyal at time $1$, as indicated with dashed black lines between the snapshots, while one vertex moves between parts.  Note also that, unlike in the case of static modularity, it is beneficial for a part to be disconnected at time $2$ as this increases the loyalty contribution.}
        \label{fig:example}
    \end{figure}
    
We sometimes consider the $k$-part temporal modularity or \emph{temporal k-modularity} of $\cG$, which is given by $q^*_{k,\omega}(\cG) = \max_{\cA, \abs{\cA} \leq k} q_\omega(\cG, \cA)$.

In our algorithms, it will often be more convenient to compute the non-normalised temporal modularity of a temporal graph (or temporal subgraph); for completeness, we now define this notion.

\begin{definition}
    Given a temporal graph $\cG = (G,\lambda) = ((V,E),\lambda)$, a partition $\cA$ of this graph, and a tuning parameter $\omega \geq 0$, the \emph{non-normalised temporal modularity} of $\cA$ is given by
    \begin{equation}
        \widetilde{q}_\omega(\cG, \cA) = \sum_{t=1}^T \sum_{A \in \cA_t} \left( 2e_{G_t}(A) - \frac{\vol_{G_t}(A)^2}{2m_t}\right) + \omega \loy(\cA).
    \end{equation}
    \end{definition}
    The \emph{non-normalised temporal modularity} of $\cG$ is given by $\widetilde{q}^{\,*}_\omega(\cG) = \max_{\cA} \widetilde{q}(\cG,\cA)$
    and 
    $\widetilde{q}^{\,*}_{k,\omega}(\cG)$$= \max_{\cA, \abs{\cA} \leq k} \widetilde{q}_\omega(\cG, \cA)$ gives the \emph{non-normalised temporal k-modularity} of $\cG$.



Our main result in this paper is an approximate fpt-algorithm  parameterised by treewidth. We use standard treewidth definitions and results as in \cite{cygan2015Algorithms}.
\begin{definition} \label{def:tree_decomp} \cite{cygan2015Algorithms}
    A \emph{tree decomposition} of a graph $G$ is a pair $(H,\mathcal{D})$ where $H$ is a tree and $\mathcal{D}=\{\mathcal{D}(b):b\in V(H)\}$ is a collection of subsets of $V(G)$, called \emph{bags}, such that:
\begin{enumerate}
    \item For every vertex $v \in V(G)$ there is at least one $b \in H$ such that $v \in \mathcal{D}(b)$.
    \item For every edge $uv \in E(G)$ there exists a node $b \in H$ such that $u,v \in \mathcal{D}(b)$.
    \item For every vertex $v$ the subgraph induced by the set of nodes $\{b: v \in \mathcal{D}(b)\}$ is connected.
\end{enumerate}
The \emph{width} $\w$ of the tree decomposition is the size of the largest bag minus 1.
\end{definition}
 The \emph{treewidth} is the minimum width over all possible tree decompositions of a graph. We use \emph{nice} tree decompositions, which have restrictions on the relationships between adjacent bags. Any tree decomposition can, in polynomial time, be transformed into a nice tree decomposition without increasing the width, and without loss of generality we can assume any nice tree decomposition of width $\w$ has $\mathcal{O}(\w n)$ nodes~\cite{cygan2015Algorithms}.

    \section{Basic properties of temporal modularity}\label{sec:basic_facts}

In this section, we record the analogues of some basic properties and approximation results from the static case for temporal modularity, and show that temporal modularity behaves in expected ways in simple cases. To begin with, if the temporal graph is actually the same graph at every time, its optimum partitions correspond exactly to the optimal partitions from the static case -- no unexpected new partitions appear, nor do any disappear.

\begin{lemma}
    For any temporal graph $\cG$ and any partition $\cA$, it holds that $-\frac{1}{2} \leq q(\cG,\cA) \leq 1$, and $q^*(\cG) \geq \frac{1}{2}\left(1 + \frac{\sum_{i=1}^T m_i}{\omega n(T-1)}\right)^{-1}$. These bounds are tight for every lifetime $T$ and every choice of tuning parameter $\omega$.
\end{lemma}
\begin{proof}
    For normal modularity, it is well known that $-\frac{1}{2} \leq q(G,\cA) \leq 1$ for all graphs $G$ and partitions $\cA$, with equality achieved for the lower bound by a complete bipartite graph with $\cA$ as the two parts of the bipartition.

    So let $\cG$ and $\cA$ be arbitrary, and we can compute that
    \begin{align*}
        q(\cG,\cA) &= \frac{1}{2\left(\sum_{i=1}^T m_i + \frac{\omega}{2}n(T-1)\right)}\left(\omega \loy(\cA) + \sum_{i=1}^T 2m_iq(G_i,\cA_i)\right)\\
        &\geq \frac{\sum_{i=1}^T 2m_iq(G_i,\cA_i)}{2\left(\sum_{i=1}^T m_i + \frac{\omega}{2}n(T-1)\right)}.
    \end{align*}
    Now, if $\sum_{i=1}^T m_i + \frac{\omega}{2}n(T-1)$ is positive, then of course the entire expression is positive and thus greater than $-1/2$. So assume this expression is negative, and then we get that
    \begin{align*}
        \frac{\sum_{i=1}^T 2m_iq(G_i,\cA_i)}{2\left(\sum_{i=1}^T m_i + \frac{\omega}{2}n(T-1)\right)} & \geq \frac{\sum_{i=1}^T 2m_iq(G_i,\cA_i)}{2\sum_{i=1}^T m_i} \geq \frac{\sum_{i=1}^T 2m_i\left(-\frac{1}{2}\right)}{2\sum_{i=1}^T m_i} = -\frac{1}{2}.
    \end{align*}

    The upper bound is immediate from that $\loy(\cA) \leq n(T-1)$ and $q(G_t,\cA_t) \leq 1$, and the lower bound on $q^*(\cG)$ is achieved by letting each part contain one single vertex at every time, so we get modularity zero at each snapshot, but achieve all loyalty edges. If each snapshot of $\cG$ has optimal modularity zero, it is easily seen that this is in fact the optimal partition, proving tightness of this bound.

    Finally, let $\cG$ be the graph that is a complete bipartite graph $K_{n,n}$ at each time with lifetime $T$, and let $\cA$ at each timestep be the two parts of the bipartition, with no part appearing at more than one time. This will mean the modularity at each snapshot is $-\frac{1}{2}$, $\loy(\cA)$ is zero, and $m_i = n^2$ for every $t$, so we can compute
    \begin{align*}
        q(\cG,\cA) &= \frac{1}{2\left(\sum_{i=1}^T m_i + \frac{\omega}{2}n(T-1)\right)}\left(\omega \loy(\cA) + \sum_{i=1}^T 2m_iq(G_i,\cA_i)\right)\\
        &= \frac{1}{2(Tn^2 + \frac{\omega}{2}n(T-1)}\left(\sum_{i=1}^T 2n^2\left(-\frac{1}{2}\right)\right)\\
        &= \left(-\frac{1}{2}\right)\left(1 + \frac{\omega}{2n}\left(1 - \frac{1}{T}\right)\right)^{-1},
    \end{align*}
    and this is easily seen to go to $-\frac{1}{2}$ as $n$ goes to infinity, keeping $T$ and $\omega$ fixed, and so the bound on $q(\cG,\cA)$ we gave is tight for every choice of $T$ and $\omega$.
    \qed
\end{proof}

For the next result, we observe that the temporal modularity can be written more compactly as
\begin{equation}\label{eqn:temp_mod_as_sum_of_static_mod}
    q(\cG, \cA) = \frac{1}{2\mu}\left(\omega \loy(\cA) + \sum_{t=1}^T 2m_tq(G, \cA_t)\right),
\end{equation}
using the definition of the static modularity at each time.

\begin{lemma}
    Suppose $\cG=G_1, \ldots, G_T$ where $G_i=G$ for each $i=1, \ldots T$. Let $\mathcal{A}^{\rm opt}_G = \{ \cA_1, \cA_2, \ldots, \cA_r \}$ be the set of optimal partitions of $V(G)$. Then there are $r$ optimal partitions of $\cG$, which are those in which one of the partitions $\mathcal{A}^{\rm opt}_G$ is repeated at each level.
\end{lemma}
\begin{proof}
    If we pick one of the partitions in $\cA^{\rm opt}_G$ and repeat it at each level, we get the maximum loyalty bonus and the optimal modularity at each time, so the total temporal modularity will be
    $$\frac{1}{2\mu}\left(\omega n(T-1) + 2Tmq^*(G)\right),$$
    as can easily be seen referring to \eqref{eqn:temp_mod_as_sum_of_static_mod}.

    There are three possible ways to pick a partition which does not repeat the same optimal partition at each level:
    \begin{enumerate}
        \item We can pick a non-optimal partition to repeat at each level. In this case, we still get the full $\omega n(T-1)$ loyalty bonus, but the other term will decrease, since we picked a suboptimal partition to repeat.
        \item We can pick optimal partitions at each level, but choose different ones at different levels. In this case, we still get the full $2Tmq^*(G)$, but because we switch partitions, some vertices fail to be loyal, and we get a smaller loyalty bonus.
        \item We can choose differing partitions at different levels, \emph{and} not choose optimal partitions at each level. In this case, both the loyalty bonus and the static modularity term are decreased.
    \end{enumerate}

Because in all three cases the overall score is decreased we have our required result.  
    
    \qed
\end{proof}

It unfortunately turns out not to be the case that the temporal modularity of such a graph is equal to the static modularity, since the tuning parameter will play a role as well, but we do get a reasonably nice formula relating the two quantities.

\begin{rmk}
    If we consider a temporal graph $\cG$ which is a copy of the same graph $G$ at every time, we get a simple expression for the modularity of this in terms of the modularity of the static graph and the number of timesteps.

    Taking the expression we found in the proof above and simplifying further, we get that
    \begin{align*}
        q^*(\cG) &= \frac{1}{2\mu}\left(\omega n(T-1) + 2Tmq^*(G)\right)\\
        &= \frac{1}{2\left(\frac{1}{2}\omega n(T-1) + \sum_{t=1}^T m_t\right)}\left(\omega n(T-1) + 2Tmq^*(G)\right)\\
        &= \frac{\omega n(T-1) + 2Tmq^*(G)}{\omega n(T-1) + 2Tm}\\
        &= \frac{\omega n + 2mq^*(G) - \frac{\omega n}{T}}{\omega n + 2m - \frac{\omega n}{T}},
    \end{align*}
    which, for large $T$, is essentially a weighted average of $1$ and $q^*(G)$, with weights $\omega n$ and $2m$ respectively.
\end{rmk}

\section{Approximations to temporal modularity}\label{sec:approx}

In this section, we prove that we can make certain simplifying assumptions and still retain a multiplicative approximation to the temporal modularity.

First we note that the modularity of a graph may also be defined as a sum over pairs of vertices. Specifically, letting $\pi_\cA(v)$ denote the part that vertex $v$ belongs to in a partition $\cA$, and letting $\delta_\cA(u,v) = \ind{\pi_\cA(u) = \pi_\cA(v)}$, we have that
$$q(G, \cA) = \frac{1}{2m} \sum_{u,v \in G} \delta_\cA(u,v)\left(\ind{uv \in E(G)} - \frac{d_u d_v}{2m}\right).$$

Using this, and our definition of loyalty, we get the following alternative form of the definition of temporal modularity:

\begin{rmk}\label{remark:sum_form_def_tempmod}
    We may equivalently define the temporal modularity score as
    \begin{equation} \label{eqn:sum_form_def_tempmod}
        q_\omega(\cG,\cA) = \sum_{(u,t), (u', t') \in V(G)\times[T]} \delta_\cA((u,t),(u',t')) \kappa_\omega(u,u',t,t')
    \end{equation}
    where
    $$\kappa_\omega(u,u',t,t') = \frac{1}{2\mu(\cG)}\left(\omega\ind{t = t'-1}\ind{u = u'} + \ind{t = t'}\left(\ind{uu' \in E_t} - \frac{d_{u,t} d_{u',t}}{2m_t}\right)\right)$$
    is a constant that depends on $\cG$ and $\omega$ but not on the partition $\cA$.
\end{rmk}

Our next result is an analogue of the result of Dinh and Thai for the static setting \cite[Lemma 3]{Dinh_Thai_2013}, which says that we can approximate the maximum modularity by considering only partitions with a fixed number of parts.  We note that the Dinh-Thai result was previously exploited by Meeks and Skerman \cite{w1hard} to obtain an efficient approximation algorithm to compute static modularity on graphs of small treewidth.

\begin{lemma}\label{lem:k_part_approx_bound}
Let $\cG$ be a temporal graph with $n$ vertices and lifetime $T$ such that each snapshot has at least one edge. Fix an integer $k\geq 2$. Then
\[ q^*_{k,\omega}(\cG)\geq \left(1-\frac{1}{k}\right) q^*_{\omega}(\cG) + \frac{n\omega(T-1)}{k}. \]
\end{lemma}

\begin{proof}
    We use the alternative definition of temporal modularity given in Remark~\ref{remark:sum_form_def_tempmod}. Let $\cA$ be an optimal partition of $\cG$ with an unlimited number of parts. We create a random partition $\BB$ of $\cG$ into $k$ parts by assigning each part of $\cA$ a random integer from $[k]$, and taking the union of all parts with the same number.

    Let us write
    $$c_{u,u',t,t'} = \omega\ind{t = t'-1}\ind{u = u'} + \ind{t = t'}\left(\ind{uu' \in E(G_t)} - \frac{d_{u,t} d_{u',t}}{2m_t}\right)$$
    so that \eqref{eqn:sum_form_def_tempmod} becomes
    $$2\mu(\cG) q(\cG,\BB) = \sum_{(u,t),(u',t') \in V(G)\times[T]} \delta_{\BB}((u,t),(u',t')) c_{u,u',t,t'},$$
    where we note that $c_{u,u',t,t'}$ does not depend on $\BB$.

    Thus, we get that
    \begin{align*}
        2\mu\Exp{q(\cG,\BB)} &= \sum_{(u,t),(u',t')\in V(G)\times [T]} \Exp{\delta_{\BB}((u,t),(u',t'))}c_{u,u',t,t'}\\
        &= \sum_{\substack{(u,t),(u',t')\in V(G)\times [T]\\\pi_\cA(u,t) = \pi_\cA(u',t')}} c_{u,u',t,t'}\\
        &\qquad + \sum_{\substack{(u,t),(u',t')\in V(G)\times [T]\\\pi_\cA(u,t) \neq \pi_\cA(u',t')}} c_{u,u',t,t'}\Prob{\pi_\BB(u,t) = \pi_\BB(u',t') \middle\vert \pi_\cA(u,t) \neq \pi_\cA(u',t')}\\
        &= \sum_{\substack{(u,t),(u',t')\in V(G)\times [T]\\\pi_\cA(u,t) = \pi_\cA(u',t')}} c_{u,u',t,t'} + \sum_{\substack{(u,t),(u',t')\in V(G)\times [T]\\\pi_\cA(u,t) \neq \pi_\cA(u',t')}} c_{u,u',t,t'}\frac{1}{k}\\
        &= \frac{1}{k}\sum_{(u,t),(u',t') \in V(G)\times [T]} c_{u,u',t,t'} + \left(1 - \frac{1}{k}\right)\sum_{\substack{(u,t),(u',t')\in V(G)\times [T]\\\pi_\cA(u,t) = \pi_\cA(u',t')}} c_{u,u',t,t'}\\
        &= \frac{1}{k}\sum_{(u,t),(u',t') \in V(G)\times [T]} c_{u,u',t,t'} + 2\mu\left(1 - \frac{1}{k}\right)q(G,\cA).
    \end{align*}

    So what remains to be done is to analyse the sum of our weights $c_{u,u',t,t'}$. We compute that
    \begin{align*}
        \sum_{(u,t),(u',t') \in V(G)\times [T]} c_{u,u',t,t'} &= \sum_{(u,t),(u',t') \in V(G)\times [T]} \omega\ind{t = t'-1}\ind{u = u'}\\
        &\qquad\qquad+ \ind{t = t'}\left(\ind{uu' \in E(G_t)} - \frac{d_{u,t} d_{u',t}}{2m_t}\right)\\
        &= n\omega(T-1) + \sum_{t=1}^T \sum_{u,u' \in V(G)} \ind{uu'\in E(G_t)} - \frac{d_{u,t} d_{u',t}}{2m_t}\\
        &= n\omega(T-1) + \sum_{t=1}^T \sum_{u\in V(G)} \left(d_{u,t} - d_{u,t}\frac{1}{2m_t}\sum_{u' \in V(G)} d_{u',t}\right)\\
        &= n\omega(T-1).
    \end{align*}
    Thus, in total we have that
    $$\Exp{q(G,\BB)} = \left(1 - \frac{1}{k}\right)q(G,\cA) + \frac{n\omega(T-1)}{2k\mu},$$
    which, since $\BB$ is by construction a partition into at most $k$ parts and $\cA$ an optimal partition, proves the result.
\qed
\end{proof}

We now prove two results that together show we can approximate the (non-normalised) temporal modularity by summing the temporal modularity of the graph restricted to suitably chosen short time windows.  We first introduce some notation for the restriction of a temporal graph to a specified time interval.  Given a temporal graph $\cG = (G,\lambda)$ with lifetime $T$, and integers $a,b$ with $1 \leq a \leq b \leq T$, we will write $\cG_{[a,b]}$ for restriction of $\cG$ to the time interval $[a,b]$; that is, $\cG_{[a,b]} = (G,\lambda')$ where, for all $e \in E(G)$, $\lambda'(e) = \lambda(e) \cap \{a,a+1,\dots,b\}$.

\begin{lemma}\label{lem:interval-approx-ub}
    Let $\cG$ be a temporal graph with lifetime $T$, and fix integers $0 = t_0 < t_1 < \cdots < t_{\ell} = T$.  Then
    $$ \sum_{i=1}^{\ell} \widetilde{q}^{\,*}(\cG_{[t_{i-1}+1,t_i]}) \leq \widetilde{q}^{\,*}(\cG).$$
\end{lemma}
\begin{proof}
    For each $1 \le i \le \ell$, let $\cA_i$ be a partition of $\cG_{[t_{i-1}+1,t_i]}$ such that $\widetilde{q}^{\,*}(\cG_{[t_{i-1}+1,t_i]}) = \widetilde{q}(\cG_{[t_{i-1}+1,t_i]},\cA_i)$.  Now define $\cA$ to be the partition of $\cG$ that agrees with each $\cA_i$ on timesteps in the range $[t_{i-1}+1,t_i]$ for each $1 \le i \le \ell$, and note that by definition of $\widetilde{q}^{\,*}(\cG)$ we have
    \[
        \widetilde{q}^{\,*}(\cG) \geq \widetilde{q}(\cG,\cA).
    \]
    Moreover, recall that
    \begin{align*}
        \widetilde{q}_\omega(\cG, \cA) &= \sum_{t=1}^T \sum_{A \in \cA_t} \left( 2e_{G_t}(A) - \frac{\vol_{G_t}(A)^2}{2m_t}\right) + \omega \loy(\cA) \\
            &\geq \sum_{i=1}^\ell \left( \sum_{t=t_{i-1}+1}^{t_i} \sum_{A \in \cA_t} \left( 2e_{G_t}(A) - \frac{\vol_{G_t}(A)^2}{2m_t}\right) + \omega \loy(\cA_i)\right) \\
            &= \sum_{i=1}^\ell \widetilde{q}(\cG_{[t_{i-1}+1,t_i]},cA_i)\\
            &= \sum_{i=1}^\ell \widetilde{q}^{\,*}(\cG_{[t_{i-1}+1,t_i]}),
    \end{align*}
    giving the result.\qed
\end{proof}

We now show that it is possible to choose values of $t_1,\dots,t_{\ell-1}$ so that each difference $t_i - t_{i-1}$ is small but we do not lose a large proportion of loyalty edges between timesteps $t_i$ and $t_{i+1}$. Recall we define the loyalty contribution as 
    $\loy(\cA) = \sum_{v \in V} \sum_{t=1}^{T-1} \delta_\cA((v,t),(v,t+1))$.
In the following lemma, we want to consider the contribution to this total from a single pair of consecutive timesteps; to this end, we define 
    $$\loy(\cA,t) = \sum_{v \in V} \delta_\cA((v,t),(v,t+1)).$$




\begin{lemma}\label{lem:interval-approx-lb}
    Let $\cG$ be a temporal graph with lifetime $T$, and let $d$ be a positive integer.  Then there exist integers $0 = t_0 < t_1 < \dots < t_\ell = T$ such that $t_j - t_{j-1} \leq 2d$ for each $1 \leq j \leq \ell$ and 
        $$ \sum_{i=0}^{\ell-1} \widetilde{q}(\cG_{[t_i+1,t_{i+1}]}) \geq \left(1 - \frac{1}{d}\right) \widetilde{q}^{\,*}(\cG).$$
\end{lemma}
\begin{proof}
    Let $\cA$ be a partition of $\cG$ such that $\widetilde{q}^{\,*}(\cG) = \widetilde{q}(\cG,\cA)$.  For each $1 \leq j \leq \lfloor \frac{T}{d} \rfloor$, fix an integer $t_j \in [(j-1)d + 1, j d]$ such that 
        $$\loy(\cA,t_j) = \min_{t \in [(j-1)d + 1, j d]} \loy(\cA,t).$$
    Note that this choice of $t_j$ ensures that
    \begin{equation}\label{lem3:min_bounded_by_av}
        \loy(\cA,t_j) \leq \frac{1}{d} \sum_{t=(j-1)d+1}^{jd} \loy(\cA,t).
    \end{equation}
    For convenience we further define $t_0 = 0$ and $t_{\lfloor \frac{T}{d} \rfloor + 1} = T$.  Note that, by construction, we have $t_{j+1} - t_j \leq 2d$ for each $0 \leq j \leq \lfloor \frac{T}{d} \rfloor + 1$.  It remains to show that 
        $$\sum_{j=0}^{\lfloor \frac{T}{d} \rfloor} \widetilde{q}^{\,*}(\cG_j) 
 \geq \left(1 - \frac{1}{d}\right) \widetilde{q}^{\,*}(\cG)$$ 
    where $\cG_j = \cG_{[t_j+1,t_{j+1}]}$ and $\cA_j$ is the restriction of $\cA$ to $\cG_j$ for each $0 \leq j \leq \lfloor \frac{T}{d} \rfloor$.

    To see this, observe that
    \begin{align*}
        \sum_{j=0}^{\lfloor \frac{T}{d} \rfloor} \widetilde{q}^{\,*}(\cG_j) &\geq \sum_{j=0}^{\lfloor \frac{T}{d} \rfloor} \widetilde{q}(\cG_j,\cA_j) \\
            &= \sum_{j=0}^{\lfloor \frac{T}{d} \rfloor} \left( \sum_{t=t_j+1}^{t_{j+1}} \sum_{A \in \cA_t} \left( 2e_{G_t}(A) - \frac{\vol_{G_t}(A)^2}{2m_t}\right) + \omega \loy(\cA_j)\right) \\
            &= \sum_{t=1}^T \sum_{A \in \cA_t} \left( 2e_{G_t}(A) - \frac{\vol_{G_t}(A)^2}{2m_t}\right) + \omega \loy(\cA) - \omega \sum_{j=1}^{\lfloor \frac{T}{d} \rfloor} \loy(\cA,t_j) \\
            &\geq \sum_{t=1}^T \sum_{A \in \cA_t} \left( 2e_{G_t}(A) - \frac{\vol_{G_t}(A)^2}{2m_t}\right) + \omega \loy(\cA) - \omega \sum_{j=1}^{\lfloor \frac{T}{d} \rfloor} \frac{1}{d} \sum_{t=(j-1)d+1}^{jd} \loy(\cA,t) 
    \end{align*}
    where the last inequality followed by~\eqref{lem3:min_bounded_by_av}. Now notice that the final part of the expression is 
    simply the sum of $\loy(\cA, t)$ over $t=1, \ldots, \lfloor T/d\rfloor$ and so is upper bounded by $\frac{\omega}{d}\loy(\cA)$ and thus,
    \begin{align*}
        \sum_{j=0}^{\lfloor \frac{T}{d} \rfloor} \widetilde{q}^{\,*}(\cG_j) 
            &\geq \sum_{t=1}^T \sum_{A \in \cA_t} \left( 2e_{G_t}(A) - \frac{\vol_{G_t}(A)^2}{2m_t}\right) + \omega \loy(\cA) - \frac{\omega}{d} \loy(\cA) \\
            &= \sum_{t=1}^T \sum_{A \in \cA_t} \left( 2e_{G_t}(A) - \frac{\vol_{G_t}(A)^2}{2m_t}\right) + \left(1-\frac{1}{d}\right)\omega \loy(\cA)\\
            &\geq \left(1 - \frac{1}{d}\right) \widetilde{q}(\cG,\cA)
            = \left(1 - \frac{1}{d}\right) \widetilde{q}^{\,*}(\cG),
    \end{align*}
    as required. \qed
\end{proof}

\section{Tree-decomposition based algorithm}\label{sec:algorithm}

In this section we describe our main algorithm, and prove that it efficiently computes an approximation to temporal modularity when the underlying graph has small treewidth.

\begin{theorem}\label{thm:tw-approx}
    Let $\mathcal{G}$ be a temporal graph with $n$ vertices and lifetime $T$, and suppose that the underlying graph of $\mathcal{G}$ has treewidth at most $\w$. Let $c$ and $d$ be positive integers. Then we can compute a $\left(1-(\frac{1}{c}+\frac{2}{d})\right)$-approximation to $q^*_\omega(\mathcal{G})$ in time $T^2n^{\mathcal{O}(cd)}c^{\mathcal{O}(\w d)}d^{\mathcal{O}(c)}$. 
\end{theorem}


As a subroutine, our approximation algorithm will use an algorithm which computes exactly the maximum temporal modularity achievable with a partition into $c$ parts when both the underlying treewidth and lifetime are small.

\begin{theorem}\label{thm:c-modularity}
Let $\cG$ be a temporal graph with $n$ vertices, lifetime $T$ and underlying treewidth at most $\w$. The non-normalised temporal $c$-modularity of $\cG$, $\widetilde{q}^{\,*}_{c,\omega}(\mathcal{G})$, can be calculated in time $n^{\mathcal{O}(cT)}c^{\mathcal{O}(\w T)}T^{\mathcal{O}(c)}$.
\end{theorem}

Theorem \ref{thm:c-modularity} is proved using fairly standard dynamic programming techniques over a \emph{nice tree decomposition}, which is defined as follows.

\begin{definition}\cite{cygan2015Algorithms}
We call a tree decomposition $(H,\mathcal{D})$ \emph{nice} if all leaves and the root of the tree contain empty bags and all non-leaf nodes are one of three types:
\begin{itemize}
    \item \emph{Introduce node}: a node $b$ with exactly one child $b'$ such that $\mathcal{D}(b)=\mathcal{D}(b')\cup \{v\}$, for some vertex $v \notin \mathcal{D}(b')$.
    \item \emph{Forget node}: a node $b$ with exactly one child $b'$ such that $\mathcal{D}(b)=\mathcal{D}(b')\setminus \{v\}$ for some vertex $v \in \mathcal{D}(b')$.
    \item \emph{Join node}: a node $b$ with exactly two children $b_1$ and $b_2$ such that $\mathcal{D}(b)=\mathcal{D}(b_1)=\mathcal{D}(b_2)$.
\end{itemize}
\end{definition}

We define $V_b$ to be the set of all vertices in $\mathcal{D}(b)$ and all bags below the node $b$.
We call the set $V_b\setminus\mathcal{D}(b)$ the set of vertices that have been \emph{forgotten} at $b$. By the third property from Definition \ref{def:tree_decomp} (i.e.~that the subtree of the tree decomposition induced by the nodes containing any given vertex is connected) we know that these vertices occur strictly below $\mathcal{D}(b)$ in the tree.

We now describe the information we will store at each node of the nice tree decomposition in our algorithm.
For any node $b \in V(H)$ a \emph{state} of $b$ consists of the following:
\begin{enumerate}
    \item A partition function $\pi : \mathcal{D}(b) \times [T] \rightarrow [c]$.
    \item A function $\alpha : [c] \rightarrow [mT]$, a count of all time-edges within each part that have at least one vertex forgotten. 
    \item A function $\beta : [c]\times[T] \rightarrow [2m]$, a count of the degree of all forgotten vertices within each time and part.
    \item A number $\gamma$, a count of the total loyalty edges from forgotten vertices.
\end{enumerate}

\begin{lemma}\label{lem:state_count}
Let $\cG$ be a temporal graph with $n$ vertices, $m$ edges, lifetime $T$ and underlying treewidth at most $\w$. 
Then there are $\mathcal{O}(2^{cT}c^{(\w +1)T}m^{c(T+1)}T^{c+1}n)$ states describing partitions into at most $c$ parts at any node in a tree decomposition of $\cG$.
\end{lemma}
\begin{proof}
    We consider all possible partition functions; for the at most $\w +1$ vertices in the bag and the $T$ times in the lifetime, there are $c^{(\w +1)T}$ ways a partition function could place them into $c$ parts. Hence, there are $c^{(\w +1)T}$ possible such functions.
    Similarly, for each part $p \in [c]$ the value $\alpha(p)$ can take is any integer from $0$ to $mT$, which gives $(mT+1)^c$ possible functions $\alpha$.
    There are $cT$ combinations of times $t \leq T$ and parts $p\in [c]$ and for each such pair $\beta(p,t)$ is an integer from $0$ to $2m$. This gives $(2m+1)^{cT}$ possible functions $\beta$.
    Finally, the loyalty score can take any integer value from $0$ where no vertex is in the same part at two consecutive times to $n(T-1)$ where all vertices are loyal at all times. This gives $n(T-1)+1$ possible values of $\gamma$.
    Therefore, there are $\mathcal{O}(2^{cT}c^{(\w +1)T}m^{c(T+1)}T^{c+1}n)$ combinations of possible functions that the state could have. \qed
\end{proof}

For a node $b \in V(H)$ a state $(\pi,\alpha,\beta,\gamma)$ of $b$ is \emph{valid} if there exists a partition function $\pi^*:V_b\times T \rightarrow [c]$ such that:
\begin{enumerate}
\item $\pi^*| _{\mathcal{D}(b)\times T}=\pi$,
    \item For each $p \in [c]$,  $\alpha(p)$ is the number of all time-edges in part $p$ in $\pi^*$ that have at least one vertex in $V_b \setminus \mathcal{D}(b)$ or equivalently $\alpha(p)=|\{(uv,t):t\in[T],uv\in E_t,\pi^*(v,t)=\pi^*(u,t)=p \text{ and } \{u,v\}\cap V_b \setminus \mathcal{D}(b)\neq \emptyset \}|$, 
    \item For each time $t$ and $p \in [c]$, $ \beta(p,t)$ is the sum of degrees of all vertices in $V_b \setminus \mathcal{D}(b)$ in part $p$ at time $t$ in $\pi^*$ or equivalently $\beta(p,t)=\sum_{\substack{u\in V_b \setminus \mathcal{D}(b)\\\pi^*(u,t)=p}} d_{u,t}$, and
    \item $\gamma$ is the loyalty contribution from all vertices in $V_b \setminus \mathcal{D}(b)$ in $\pi^*$, equivalently  $\gamma= \sum_{v \in V_b \setminus \mathcal{D}(b)} \sum_{t=1}^{T-1} \delta_{\pi^*}((v,t),(v,t+1))$.
\end{enumerate}
If these conditions hold we say that the partition function $\pi^*$ \emph{supports} the valid state $(\pi,\alpha,\beta,\gamma)$.  

We now describe how to determine the set of valid states for each type of node, given the sets of valid states for all of the node's children.

\begin{lemma} \label{lem:tw leaf}
    A state $(\pi,\alpha,\beta,\gamma)$ at a leaf node is valid if and only if:
    \begin{itemize}
        \item $\pi$ is the empty partition function,
        \item $\alpha(p)=0$ for all $p\in[c]$,
        \item $\beta(p,t)=0$ for all $p \in [c],t\in[T]$, and
        \item $\gamma=0$.
    \end{itemize}
\end{lemma}
\begin{proof}
From our definition of a nice tree decomposition we know that all leaf nodes have empty bags, hence the only possible partition function is the empty partition function.
If $(\pi,\alpha,\beta,\gamma)$ is a valid state at the empty leaf node $b$ then there is some partition function $\pi^*$ on the forgotten vertices that supports it. As $b$ is a leaf node there are no nodes that occur below $b$ and therefore no vertices forgotten at $b$. Hence $\pi^*$ must be the empty partition function. As no vertices are forgotten the count of degrees, edges and loyalty edges forgotten must all be zero, that is $\alpha(p)=0$, $\beta(p,t)=0$ and $\gamma=0$ for all $p\in[c]$ and $t\in[T]$.
Conversely, if $(\pi,\alpha,\beta,\gamma)$ is as given in the lemma statement then the empty partition $\pi^*$ supports the state and therefore it is valid.
\qed 
\end{proof}

\begin{lemma} \label{lem:tw intro}
    Let $b$ be an introduce node with child $b'$ such that $\mathcal{D}(b)\setminus\mathcal{D}(b')=\{v\}$. Then $(\pi,\alpha,\beta,\gamma)$ is a valid state for $b$ if and only if there exists a valid state  $(\pi',\alpha',\beta',\gamma')$ for $b'$ such that:
    \begin{enumerate}
    \item $\pi| _{\mathcal{D}(b')\times [T]}=\pi'$,
        \item $\alpha(p)=\alpha'(p)$ for all $p\in [c]$,
        \item $ \beta(p,t)=\beta'(p,t)$ for all $p \in [c], t\in [T]$, and
        \item $\gamma=\gamma'$.
    \end{enumerate} 
\end{lemma}
\begin{proof}
 Suppose $(\pi,\alpha,\beta,\gamma)$ is a valid state for $b$, and let $\pi^*$ be a partition function supporting its validity. We will show that the partition function $\pi^{**}:=\pi^*|_{(V_b\times T)\setminus (\{v\}\times T)}$ supports the state $(\pi',\alpha',\beta',\gamma')$.
 Note that as no new vertices have been forgotten $V_b\setminus \mathcal{D}(b)=V_{b'}\setminus \mathcal{D}(b')$, that is the forgotten sets of each node are equal.
 
Let $\pi'$ be the restriction of $\pi^{**}$ onto $\mathcal{D}(b')\times T$, then as $\mathcal{D}(b')\subset \mathcal{D}(b)$ we have that $\pi'$ is the partition function corresponding to $\pi^{**}$ with $\pi| _{\mathcal{D}(b')\times T}=\pi'$ as required.
    
For $p \in [c]$, $\alpha(p)$ is the count of all time-edges in part $p$ in $\pi^*$ that have at least one vertex in $V_b\setminus \mathcal{D}(b)$ and $\alpha'(p)$ is the count of time-edges in part $p$ in $\pi^{**}$ that have at least one vertex in $V_{b'} \setminus \mathcal{D}(b')$. As $V_b\setminus \mathcal{D}(b)=V_{b'}\setminus \mathcal{D}(b')$ and $\pi^{*}$ and $\pi^{**}$ are the same on the forgotten sets we have that $\alpha'$ is the count of time-edges corresponding to $\pi^{**}$  with $\alpha(p)=\alpha'(p).$
    
For $p \in [c]$ and time $t\in [T]$,  $\beta(p,t)$ is the count of the degree of vertices in $V_{b}\setminus \mathcal{D}(b)$ at time $t$ in part $p$ in $\pi^*$. As the forgotten sets are identical this count is the same for vertices in $V_{b'}\setminus \mathcal{D}(b')$ in $\pi^{**}$ which is the definition of $\beta'(p,t)$.
Hence, $\beta'$ is the count of degree of forgotten vertices corresponding to $\pi^{**}$ with $\beta(p,t)=\beta'(p,t)$ for all $p \in [c], t \in [T]$.
    
The value $\gamma$ is the total loyalty edges in $V_{b}\setminus \mathcal{D}(b)$ in $\pi^*$, since the partition functions on the forgotten sets are identical this is the same as the total loyalty edges in $V_{b'}\setminus \mathcal{D}(b')$ in $\pi^{**}$ so $\gamma'$ is the loyalty count corresponding to $\pi^{**}$ with $\gamma=\gamma'$.
Therefore, $(\pi',\alpha',\beta',\gamma')$ is also a valid state and is supported by $\pi^{**}$.

 Suppose $(\pi',\alpha',\beta',\gamma')$ is a valid state for $b'$ such that our conditions hold, and let $\pi'^*$ be a partition function supporting its validity. 
     We will show that there is a partition function $\pi^*$ that supports the state $(\pi,\alpha,\beta,\gamma)$.
    We know from condition $1$ that $\pi| _{\mathcal{D}(b')\times [T]}=\pi'$, so it is possible to construct the partition function $\pi^*$ on $V_b \times [T]$ such that $\pi^*|_{V_{b'}\times[T]}=\pi'^*$ and $\pi^*(v,t)=\pi(v,t)$ for all $v \in \mathcal{D}(b)$. Therefore, $\pi$ is the valid partition function corresponding to $\pi^{*}$.

From condition $2$ we have that for any $p \in [c]$, $\alpha'(p)=\alpha(p)$. 
    From the validity of the child we know that $\alpha'(p)$ is the count of all time-edges in part $p$ in $\pi'^*$ that have at least one vertex in $V_{b'}\setminus \mathcal{D}(b')$, this is equal to the count of time-edges with at least one vertex in $V_{b}\setminus \mathcal{D}(b)$ in $p$ in $\pi^*$ as the partition functions are identical over these sets.
    Therefore, $\alpha(p)$ is the accurate count of time edges corresponding to $\pi^*$.
    
    From the validity of the child we know that for any $p \in [c]$ and $t \in [T]$, $\beta'(p,t)$ is the count of the degree of vertices in $V_{b'}\setminus \mathcal{D}(b')$ at time $t$ in part $p$ in $\pi'^*$. From condition $3$ we have that $\beta'(p,t)=\beta(p,t)$. As the partition functions $\pi'^*$ and $\pi^*$ are identical on the forgotten sets this count is the same for vertices in $V_{b}\setminus \mathcal{D}(b)$ in $\pi^*$.
    Hence, $\beta(p,t)$ is the accurate count of degree corresponding to $\pi^*$ as required.
    
The value $\gamma'$ is the total loyalty edges in $V_{b'}\setminus \mathcal{D}(b')$ in $\pi'^*$ and from condition $4$ we know that $\gamma'=\gamma$. Since the partition functions are identical on the forgotten sets this is the same as the total loyalty edges in $\pi^*$ so $\gamma$ is the accurate count of loyalty corresponding to $\pi^*$.
   Therefore, $(\pi',\alpha',\beta',\gamma')$ being a valid state implies $(\pi,\alpha,\beta,\gamma)$ is a valid state.
\qed
\end{proof}

\newpage

\begin{lemma} \label{lem:tw forget}
    Let $b$ be a forget node with child $b'$ such that $\mathcal{D}(b)=\mathcal{D}(b')\setminus\{v\}$ for some $v \in \mathcal{D}(b')$. Then $(\pi,\alpha,\beta,\gamma)$ is a valid state for $b$ if and only if there exists a valid state $(\pi',\alpha',\beta',\gamma')$ of $b'$ such that:
    \begin{enumerate}
    \item $\pi=\pi'| _{\mathcal{D}(b)\times [T]}$ 
    \item $\alpha(p) =\alpha'(p) +|\{(u,t):uv \in E_t,u\in \mathcal{D}(b), t\in[T] \text{ and } (u,t),(v,t) \in p\}|$ for all $p\in[c]$,
    \item for all $p\in[c]$, $t\in[T]$ 
    $$\beta(p,t)=
    \begin{cases}
        \beta'(p,t) + d_{v,t} & \text{if } \pi'(v,t)=p\\
        \beta'(p,t) & \text{otherwise,}
    \end{cases}$$ and 
    \item $\gamma = \gamma' + |\{t \in [T-1]:\pi'(v,t)=\pi'(v,t+1)\}|$.
    \end{enumerate}
\end{lemma}
\begin{proof}
    
    Suppose $(\pi,\alpha,\beta,\gamma)$ is a valid state for $b$ and let $\pi^*$ be a partition function supporting its validity.
    We will show that $\pi^*$ also supports the state $(\pi',\alpha',\beta',\gamma')$.
    Let $\pi'$ be the restriction of $\pi^*$ onto $\mathcal{D}(b')\times [T]$, then as $\mathcal{D}(b) \subset \mathcal{D}(b')$ we have that $\pi=\pi'| _{\mathcal{D}(b)\times [T]}$.

    For $p \in [c]$, $\alpha(p)$ is the count of all time-edges in part $p$ in $\pi^*$ that have at least one vertex in $V_b\setminus \mathcal{D}(b)$.
    The vertex $v$ is forgotten by $b$, hence the number of edges counted for each part by $\alpha$ includes all edges to $v$. 
    As $v$ is not forgotten by $b'$, for each part $p$ in $\pi^*$ the number of forgotten time-edges with at least one vertex in $V_{b'} \setminus \mathcal{D}(b')$ can be given by subtracting the number of edges within $p$ incident to $v$ across all times that $v$ is in $p$. So $\alpha'$ is the valid count corresponding to $\pi^*$ when $\alpha'(p) =\alpha(p) -|\{(u,t):uv \in E_t, u \in \mathcal{D}(b) \text{ and } (u,t),(v,t) \in p \}|$.
    
    For a given part $p$ and time $t$, $\beta(p,t)$ is the count of the degree for vertices in $V_b\setminus \mathcal{D}(b)$ in part $p$ at time $t$ in $\pi^*$. As the difference between the forgotten sets is exactly $\{v\}$, the degree of vertices in the set $V_{b'}\setminus \mathcal{D}(b')$ is found by subtracting the degree of $v$ at time $t$ if it was in part $p$ at that time. If $v$ is not in the part in the child vertex then the vertices in the part are unchanged and the degree is also. So $\beta'$ is the valid count corresponding to $\pi^*$ when $\beta(p,t) = \beta'(p,t) + d_{v,t} \text{ if } \pi'(v,t)=p$ and $\beta(p,t)=\beta'(p,t)$ otherwise.
    
    Similarly, $\gamma$ is the loyalty count of vertices in $V_b\setminus \mathcal{D}(b)$ in $\pi^*$, the loyalty count of vertices in $V_{b'}\setminus \mathcal{D}(b')$ is found by subtracting the loyalty count of $v$, that is the sum of all consecutive times where $v$ is in the same part. So $\gamma'$ is the accurate loyalty count corresponding to $\pi^*$ when $\gamma = \gamma' + |\{t \in [T-1]:\pi'(v,t)=\pi'(v,t+1)\}|$.
    Hence, if $(\pi,\alpha,\beta,\gamma)$ is a valid state so is $(\pi',\alpha',\beta',\gamma')$.

Suppose $(\pi',\alpha',\beta',\gamma')$ is a valid state for $b'$ such that our conditions hold and let $\pi'^*$ be a partition function supporting its validity.
We will show that $\pi'^*$ supports the state $(\pi,\alpha,\beta,\gamma)$.

We know from condition $1$ that $\pi=\pi'| _{\mathcal{D}(b)\times [T]}$ and from validity of the child that $\pi'^*|_{\mathcal{D}(b')\times [T]}=\pi'$ therefore as $\mathcal{D}(b) \subset \mathcal{D}(b')$ we have that $\pi'^*|_{\mathcal{D}(b)\times [T]}=\pi$, that is that $\pi$ is the valid partition function corresponding to $\pi^*$.

 As no new vertices are introduced between the child and parent bags, $V_b=V_{b'}$ and the set of vertices forgotten by $b$ is $V_b\setminus\mathcal{D}(b)=V_b\setminus\{\mathcal{D}(b')\cup \{v\}\}$.

From condition $2$ we have that $\alpha(p) =\alpha'(p) +|\{(u,t):uv \in E_t, u \in \mathcal{D}(b) \text{ and } (u,t),(v,t) \in p \}|$ and from validity of the child we know that for $p \in [c]$, $\alpha'(p)$ is the count of all time-edges in part $p$ in $\pi'^*$ that have at least one vertex in $V_{b'}\setminus\mathcal{D}(b')$. 
The difference in the number of time-edges with at least one vertex in $V_{b'}\setminus\mathcal{D}(b')$ is given by the number of time-edges forgotten in $p$ when $v$ is forgotten. That is exactly the number of edges incident to $v$ within $p$ in $\pi'$, which is $|\{(u,t):uv \in E_t, u \in \mathcal{D}(b) \text{ and } (u,t),(v,t) \in p \}|$.
Therefore, $\alpha$ is the accurate count of edges forgotten in each part corresponding to $\pi^*$.

From condition $3$ we have that $\beta(p,t) = \beta'(p,t) + d_{v,t} \text{ if } \pi'(v,t)=p$ and $\beta(p,t)=\beta'(p,t)$ otherwise and from validity of the child we know that for $p \in [c]$ and a time $t \in [T]$, $\beta'(p,t)$ is the count of the degree of vertices in $V_{b'}\setminus\mathcal{D}(b')$ in part $p$ at time $t$ in $\pi'^*$.
The difference in the degree of vertices in $\pi'^*$ in the set $V_{b'}\setminus\mathcal{D}(b')$ is given by the degree of vertices in this part and time forgotten between $b'$ and $b$, that is exactly the degree of $v$ if $v$ is in part $p$ at time $t$ in $\pi'$ and zero otherwise. Hence, $\beta$ is the accurate count of degrees of forgotten vertices in each part and time corresponding to $\pi^*$.

Finally, from condition $4$ we have that $\gamma = \gamma' + |\{t \in [T-1]:\pi'(v,t)=\pi'(v,t+1)\}|$ and from validity of the child we know $\gamma'$ is the loyalty contribution in $\pi'^*$ in $V_{b'}\setminus\mathcal{D}(b')$.
The change in loyalty count to the set $V_{b}\setminus\mathcal{D}(b)$ is the total loyalty contribution that $v$ gives, that is the count of all consecutive times that $v$ stays in the same part in $\pi'$. Hence $\gamma$ is the accurate loyalty count corresponding to $\pi^*$.
Therefore, if $(\pi',\alpha',\beta',\gamma')$ is a valid state so is $(\pi,\alpha,\beta,\gamma)$.  
\qed
\end{proof}

\begin{lemma} \label{lem:tw join}
    Let $b$ be a join node with children $b_1$ and $b_2$ such that $\mathcal{D}(b)=\mathcal{D}(b_1)=\mathcal{D}(b_2)$. Then $(\pi,\alpha,\beta,\gamma)$ is a valid state of $b$ if and only if there exist valid states $(\pi_1,\alpha_1,\beta_1,\gamma_1)$ of $b_1$ and $(\pi_2,\alpha_2,\beta_2,\gamma_2)$ of $b_2$ such that:
    \begin{enumerate}
        \item $\pi(v,t)=\pi_1(v,t)=\pi_2(v,t)$ for all $v\in \mathcal{D}(b)$ and $t\in [T]$,
        \item $\alpha(p)=\alpha_{1}(p)+\alpha_2(p)$ for all $p \in[c]$,
        \item $\beta(p,t)=\beta_1(p,t)+\beta_2(p,t)$ for all $p \in [c],t\in [T]$, and
        \item $\gamma=\gamma_1+\gamma_2$.
    \end{enumerate}
\end{lemma}
\begin{proof}
First note that, by the properties of tree decomposition, the set of vertices forgotten by each of the children are completely disjoint. The set of vertices forgotten by $b$ is the union of these disjoint sets.
    
Suppose $(\pi,\alpha,\beta,\gamma)$ is a valid state for $b$ and let $\pi^*$ be a partition function supporting its validity. 
We will show that there exists a state $(\pi_1,\alpha_1,\beta_1,\gamma_1)$ supported by $\pi^*_1:=\pi^*|_{V_{b1} \times [T]}$ and a state $(\pi_2,\alpha_2,\beta_2,\gamma_2)$ supported by $\pi^*_2:=\pi^*|_{V_{b2} \times [T]}$. 
    As $\mathcal{D}(b)$ is the same for all three nodes the restrictions of $\pi^*$, $\pi^*_1$ and $\pi^*_2$ to $\mathcal{D}(b) \times [T]$ are all equal. Hence, $\pi_1$ and $\pi_2$ are the partition functions corresponding to $\pi^*_1$ and $\pi^*_2$.
    
    For a part $p$, $\alpha(p)$ is the count of all time-edges in part $p$ in $\pi^*$ that have at least one vertex in $V_b\setminus \mathcal{D}(b)$. 
    No new vertices are forgotten between $b_1$, $b_2$ and $b$ so all forgotten edges must have been forgotten by $b_1$ or $b_2$. Let $\alpha_1$ count the edges that have been forgotten by $b_1$ and let $\alpha_2$ count the edges that have been forgotten by $b_2$.
    As these are disjoint, $\alpha_1$ and $\alpha_2$ are the accurate time-edge counts corresponding to $\pi^*_1$ and $\pi^*_2$ with $\alpha(p)=\alpha_{1}(p)+\alpha_2(p)$.
    
    For each part-time pair $(p,t)$ we have that $\beta(p,t)$ is the count of the degree of vertices in $V_b\setminus \mathcal{D}(b)$ in part $p$ at time $t$ in $\pi^*$.
    No vertices are forgotten by $b$ that were not already forgotten by one of $b_1$ or $b_2$. Let $\beta_1$ count the degrees of vertices forgotten by $b_1$ and let $\beta_2$ count the degrees of vertices forgotten by $b_2$.
    The vertices forgotten by $b_1$ and $b_2$ are disjoint; hence, $\beta_1$ and $\beta_2$ are the accurate degree counts corresponding to $\pi^*_1$ and $\pi^*_2$ with $\beta(p,t)=\beta_1(p,t)+\beta_2(p,t)$.
    
    Finally, $\gamma$ gives the loyalty of all vertices forgotten by $b$ in $\pi^*$. Let $\gamma_1$ be the loyalty of vertices forgotten by $b_1$ and $\gamma_2$ be the loyalty of vertices forgotten by $b_2$. Again we have that no new vertices are forgotten and no vertex can be forgotten by both $b_1$ and $b_2$.
    Therefore, $\gamma_1$ and $\gamma_2$ are the accurate loyalty counts corresponding to $\pi^*_1$ and $\pi^*_2$ with $\gamma=\gamma_1+\gamma_2$.

    Suppose $(\pi_1,\alpha_1,\beta_1,\gamma_1)$ is a valid state for $b_1$ and $(\pi_2,\alpha_2,\beta_2,\gamma_2)$ is a valid state for $b_2$ such that our conditions hold and let $\pi^*_1$ and $\pi^*_2$ be partition functions that support the respective states. We want to show that there is a partition function that supports $(\pi,\alpha,\beta,\gamma)$.

    From condition $1$ we have that $\pi(v,t)=\pi_1(v,t)=\pi_2(v,t)$ for all $v\in \mathcal{D}(b)$ and $t\in [T]$ and from validity of the children we know that $\pi^*_1|_{\mathcal{D}(b_1)\times [T]}=\pi_1$ and $\pi^*_2|_{\mathcal{D}(b_2)\times [T]}=\pi_2$.
    We define $\pi^*$ as the partition function given by $\pi^*(v,t)=\pi_1^*(v,t)$ if $v \in V_{b_1} $ and $\pi^*(v,t)=\pi_2^*(v,t)$ if $V_{b_2} \setminus \mathcal{D}(b)$. Let $\pi=\pi^*| _{\mathcal{D}(b)\times [T]}$. Then $\pi$ is the valid partition function corresponding to $\pi^*$.

    From condition $2$ we have that $\alpha(p)=\alpha_{1}(p)+\alpha_2(p)$ and from validity of the children we have that $\alpha_1(p)$ is the count of time edges in part $p$ in $\pi^*_1$ with at least one vertex in $V_{b_1} \setminus \mathcal{D}(b_1)$ and $\alpha_2(p)$ is count of time edges in part $p$ in $\pi^*_2$ with at least one vertex in $V_{b_2} \setminus \mathcal{D}(b_2)$.
By definition, $\pi^*$ is equal to $\pi^*_1$ on all vertices forgotten by $b_1$ and is equal to $\pi^*_2$ on all vertices forgotten by $b_2$. All vertices forgotten by $b$ are forgotten by exactly one of $b_1$ or $b_2$ and there are no edges between vertices forgotten by both $b_1$ and $b_2$. 
Hence, for a part $p$, the number of edges in part $p$ that have been forgotten is exactly the sum of those forgotten in the same part by $b_1$ and $b_2$.
Therefore, $\alpha$ is the valid count of forgotten time-edges corresponding to $\pi^*$.

    For any time and part pair $(p,t)$ we know from validity of the children that the degree of all vertices forgotten in part $p$ at time $t$ in $\pi^*_1$ by $b_1$ is $\beta_1(p,t)$ and the number forgotten in $\pi^*_2$ by $b_2$ is $\beta_2(p,t)$. Condition $3$ gives us that $\beta(p,t)=\beta_1(p,t)+\beta_2(p,t)$ for all $p \in [c],t\in [T]$. 
    By definition $\pi^*$ is equal to $\pi^*_1$ on all vertices forgotten by $b_1$ and is equal to $\pi^*_2$ on all vertices forgotten by $b_2$. All vertices forgotten by $b$ are forgotten by exactly one of $b_1$ or $b_2$ hence for a part $p$ and time $t$ the degree of vertices forgotten in part $p$ at time $t$ in $\pi^*$ by $b$ is exactly the sum of those forgotten in the same part and time by $b_1$ and $b_2$.
    Therefore, $\beta$ is the accurate count of forgotten degrees corresponding to $\pi^*$.
    
    Finally, from validity of the children we have that $\gamma_1$ and $\gamma_2$ are accurate loyalty counts for $\pi^*_1$ and $\pi^*_2$ respectively. Condition $4$ gives us that $\gamma=\gamma_1+\gamma_2$. Each vertex forgotten by $b$ in $\pi^*$ is forgotten by exactly one of $b_1$ in $\pi^*_1$ or $b_2$ in $\pi^*_2$ and $\pi^*$ matches $\pi^*_1$ and $\pi^*_2$ on these forgotten sets. Hence, the loyalty count of $b$ in $\pi^*$ is exactly the sum of loyalty counts for $b_1$ in $\pi^*_1$ and $b_2$ in $\pi^*_2$. Therefore $\gamma$ is the loyalty count corresponding to $\pi^*$.
    Hence $(\pi,\alpha,\beta,\gamma)$ is a valid state for $b$. \qed
\end{proof}

We also bound the time required to compute the maximum non-normalised temporal $c$-modularity of the graph given all valid states at the root.

\begin{lemma}\label{lem:tw root}
Let $\cG$ be a temporal graph with $n$ vertices, $m$ edges and lifetime $T$.
Given the set of all valid states at the root of a nice tree decomposition of $\cG$, the maximum non-normalised temporal $c$-modularity of $\cG$, $\widetilde{q}^{\,*}_c(\cG)$, can be found in time $O(2^{cT}m^{c(T+1)}T^{c+2}cn)$. 
\end{lemma}
\begin{proof}
We claim that
\begin{equation}\label{eqn:root_lemma}
\widetilde{q}^{\,*}_c(\cG)=\max_{\text{valid state at $b$ }(\pi,\alpha,\beta,\gamma)} \left(\sum_{p \in [c]} \left( 2\alpha(p) -\sum_{t=1}^T  \frac{\beta(p,t)^2}{2m_t}\right) + \omega \gamma\right).
\end{equation}
First we want to show that 
\[\widetilde{q}^{\,*}_c(\cG)\geq\max_{\text{valid state at $b$ }(\pi,\alpha,\beta,\gamma)} \left(\sum_{p \in [c]} \left( 2\alpha(p) -\sum_{t=1}^T  \frac{\beta(p,t)^2}{2m_t}\right) + \omega \gamma\right).\]
To do this we consider the valid state $(\pi,\alpha,\beta,\gamma)$ where the maximum is obtained. As the state is valid, it is supported by a partition function $\pi^*$ on all forgotten vertices. We know at the empty root node that all vertices have been forgotten, hence $\pi^*$ is a partition function on all vertices on the graph. 
We show that the sum we maximise in Equation \ref{eqn:root_lemma} is therefore the modularity of the partition given by $\pi^*$. 
By validity $\alpha$ exactly counts the edges within each part. This means for each part $p$ we have that $\alpha(p)=\sum_{t=1}^Te_{G_t}(p)$. Therefore, we have $\sum_{p \in [c]}2\alpha(p)=\sum_{p\in[c]} 2\sum_{t=1}^T e_{G_t}(p)=\sum_{t=1}^T \sum_{p \in [c]} 2e_{G_t}(p))$.
    We know from validity that $\beta$ exactly counts the degree in each part and time. This means for any part $p$ and time $t$ $\beta(p,t)=\vol_{G_t}(p)$.
    Similarly $\gamma$ is exactly the loyalty count, so $\gamma=\loy(\pi^*)$. 
\begin{align*}
\left(\sum_{p \in [c]} \left( 2\alpha(p) -\sum_{t=1}^T  \frac{\beta(p,t)^2}{2m_t}\right) + \omega \gamma\right)\\
 =\sum_{t=1}^T \sum_{p \in [c]} \left( 2e_{G_t}(p) - \frac{\text{vol}_{G_t}(p)^2}{2m_t}\right) + \omega \loy(\pi^*)\\
=\widetilde{q_c}(\cG,\pi^*) \leq \widetilde{q}^{\,*}_c.
\end{align*}

Next we show that
\[\widetilde{q}^{\,*}_c(\cG)\leq\max_{\text{valid state at $b$ }(\pi,\alpha,\beta,\gamma)} \left(\sum_{p \in [c]} \left( 2\alpha(p) -\sum_{t=1}^T  \frac{\beta(p,t)^2}{2m_t}\right) + \omega \gamma\right)\]
We know that at each step every possible combination of functions a state could take is considered and accepted as valid if it is supported by some partition function $\pi^*$ on the forgotten vertices. This means that, at a given node, every possible partition function on the forgotten vertices has a corresponding valid state. Therefore, at the root node where all vertices are forgotten every possible partition function on the temporal graph has a corresponding valid state.

Consider a valid state at the root $(\pi,\alpha,\beta,\gamma)$.
With a supporting partition function $\pi^*$ acting on all the vertices of the graph.
We start from the definition of non-normalised temporal $c$-modularity on the partition $\pi^*$:
    \[
        \widetilde{q}_\omega(\cG,\pi^*) = \sum_{t=1}^T \sum_{p \in [c]} \left( 2e_{G_t}(p) - \frac{\text{vol}_{G_t}(p)^2}{2m_t}\right) + \omega \loy(\pi^*)
    \]
    By validity $\alpha$ exactly counts the edges within each part. This means for each part $p$ we have that $\alpha(p)=\sum_{t=1}^Te_{G_t}(p)$. Therefore, we have $\sum_{t=1}^T \sum_{p \in [c]} 2e_{G_t}(p))=\sum_{p\in[c]} 2\sum_{t=1}^T e_{G_t}(p)$$=\sum_{p \in [c]}2\alpha(p)$.
    We know from validity that $\beta$ exactly counts the degree in each part and time. This means for any part $p$ and time $t$ $\beta(p,t)=\vol_{G_t}(p)$.
    Similarly $\gamma$ is exactly the loyalty count, so $\gamma=\loy(\pi^*)$.
    This gives non-normalised temporal modularity of the partition function $\pi^*$ as:
    \[
        \widetilde{q}_\omega(\cG,\pi^*) = \sum_{p \in [c]} \left( 2\alpha(p) -\sum_{t=1}^T  \frac{\beta(p,t)^2}{2m_t}\right) + \omega \gamma
    \]
The non-normalised temporal modularity of a partition can be clearly calculated given a valid state, as we are considering all possible partitions this will include the maximum which is equal to $\widetilde{q}^{\,*}_c(\cG)$.
    
For each state the non-normalised temporal modularity of a supporting partition can be calculated in time $\mathcal{O}(cT)$.
To find the non-normalised temporal $c$-modularity of the graph we compare the calculated modularity of all root states. All the vertices have been forgotten so the partition function $\pi$ of the valid state is always empty, therefore there are $(mT+1)^c(2m+1)^{cT}(nT+1)$ possible valid states. Hence at the root the maximum non-normalised temporal $c$-modularity can be found in $\mathcal{O}(2^{cT}m^{c(T+1)}T^{c+2}cn)$.
\qed
\end{proof}

We now have everything we need to prove Theorem \ref{thm:c-modularity}.

\begin{proof}[Proof of Theorem~\ref{thm:c-modularity}]
    By Lemma \ref{lem:state_count} the maximum number of states at any node is $\mathcal{O}(2^{cT}c^{(\w +1)T}m^{c(T+1)}T^{c+1}n)$.

    We assume that all our leaf nodes are empty and from Lemma \ref{lem:tw leaf} the states at the leaf nodes are valid only when the partition is empty and the other functions are zero which can checked in constant time. To compute the valid states at the root we move up the graph starting at the leaf nodes until we reach the root where we will use Lemma \ref{lem:tw root} to calculate the non-normalised temporal $c$-modularity. To do this we consider our three types of parent child relationship and the time needed to find all valid states for the parent in each case.

    First we consider an introduce node.
    Considering a state $(\pi,\alpha,\beta,\gamma)$ of the parent and a valid state $(\pi',\alpha',\beta',\gamma')$ of the child, by Lemma \ref{lem:tw intro} we can check the validity of the parent by checking the equality of each of our functions.
    For the partition function we have to check that, for each time $t\leq T$ and vertex $u\neq v$ in the bag, $\pi(u,t)=\pi'(u,t)$, which takes at most $(\w +1)T$ checks with each check taking constant time. 
    For $\alpha$ we have to check that, for each part $p\in[c]$, we have $\alpha(p)=\alpha'(p)$, which takes $c$ constant time checks.
    For $\beta$ we have to check that, for each part $p\in[c]$ and time $t\leq T$, we have $\beta(p,t)=\beta'(p,t)$, which takes $cT$ constant time checks.
    Finally, we have a single constant time check to see that $\gamma=\gamma'$.
    Therefore, checking the validity of a parent state can be done in time $\mathcal{O}(T(\w +c))$.
    We want to check each possible pair of a parent and a child state, that is $n^{\mathcal{O}(cT)}c^{\mathcal{O}(\w T)}T^{\mathcal{O}(c)}$.

    Next we look at the time taken to check a forget node. Considering a state $(\pi,\alpha,\beta,\gamma)$ for the parent and a valid state $(\pi',\alpha',\beta',\gamma')$ for the child, by  Lemma~\ref{lem:tw forget} we can check the validity by checking the relationship of each of our functions.
    For the partition function, we have to check for each time $t\in [T]$ and each vertex $u \neq v$ in the bag that we have $\pi(u,t)=\pi'(u,t)$, which takes $\w T$ constant time checks. 
    For $\alpha$, for each of the parts $p\in[c]$ we count the number of edges from $v$ in the part in the child and then check this against the difference in our $\alpha$ values, which can be done in time $\mathcal{O}(c\w ^2T)$.
    For $\beta$, for each part $p\in[c]$ and time $t\in[T]$ we count the degree of $v$ in part $p$ at time $t$ in the child then check this is equal to $\beta(p,t)-\beta'(p,t)$, which can be done in time $\mathcal{O}(cT\w )$.
    Finally, for $\gamma$ we count the loyalty of the vertex $v$ and compare the difference in our loyalty scores to this value, which can be done in time $\mathcal{O}(T)$.
    Therefore checking the validity of a parent state can be done in time $\mathcal{O}(c\w ^2T)$. 
    We need to check all possible combinations of parent and child states, and there are $O\big((2^{cT}c^{(\w +1)T}m^{c(T+1)}T^{c+1}n)^2\big)$
     such combinations. This gives us total time $n^{\mathcal{O}(cT)}c^{\mathcal{O}(\w T)}T^{\mathcal{O}(c)}$.

    Finally, we consider the time needed to check that a join node is valid. Consider a parent state ($\pi,\alpha,\beta,\gamma)$ and valid states ($\pi_1,\alpha_1,\beta_1,\gamma_1)$ and ($\pi_2,\alpha_2,\beta_2,\gamma_2)$ for the two children. From Lemma \ref{lem:tw join} we know the relationship needed between each of the four functions for the parent state to be valid, we check these relationships.
    First, to check the partition function compatibility, for each vertex $v$ in the bag and time $t$ we check that $\pi(v,t)=\pi_1(v,t)=\pi_2(v,t)$ which takes time $\mathcal{O}(\w+1)T$ checks.
    Then, for each part $p \in [c]$ we check that $\alpha(p)=\alpha_1(p)+\alpha_2(p)$ which takes time $\mathcal{O}(c)$.
    Similarly, for each part $p\in[c]$ and time $t\in[T]$ we check the value of $\beta(p,t)$ by comparing it to the sum of $\beta_1(p,t)$ and $\beta_2(p,t)$ which takes time $\mathcal{O}(cT)$.
    Finally, we check the loyalty $\gamma$ by comparing it to the sum of $\gamma_1$ and $\gamma_2$, which is a single constant time check.
    Therefore, checking the validity of the parent takes time $\mathcal{O}(T(\w+c))$.
    We need to check all possible combinations of two children and one parent state, and there are $O\big((2^{cT}c^{(\w+1)T}m^{c(T+1)}T^{c+1}n)^3\big)$ such combinations. This gives us total time $n^{\mathcal{O}(cT)}c^{\mathcal{O}(\w T)}T^{\mathcal{O}(c)}$.

Considering all possible types of node, we see that, for any node, given the valid states of its children, we can compute the valid states at the node in time $n^{\mathcal{O}(cT)}c^{\mathcal{O}(\w T)}T^{\mathcal{O}(c)}$.  We can assume any nice tree decomposition has at most $\mathcal{O}(\w n)$ nodes\cite{cygan2015Algorithms} so using our worst case complexity for each of these nodes, we can calculate all valid states at the root in $n^{\mathcal{O}(cT)}c^{\mathcal{O}(\w T)}T^{\mathcal{O}(c)}$.  From Lemma \ref{lem:tw root} the non-normalised temporal $c$-modularity at the root can be calculated in time $\mathcal{O}(2^{cT}m^{c(T+1)}T^{c+2}cn) \in n^{\mathcal{O}(cT)}T^{\mathcal{O}(c)}c$.  Hence, the non-normalised temporal $c$-modularity for a graph of bounded tree width can be calculated in time $n^{\mathcal{O}(cT)}c^{\mathcal{O}(\w T)}T^{\mathcal{O}(c)}$.\qed
\end{proof} 
We are now ready to prove Theorem~\ref{thm:tw-approx}.  The strategy is to apply the algorithm of Theorem~\ref{thm:c-modularity} to the restriction of our input graph to constant-length time intervals.  We use a second simple dynamic program to determine the best set of short intervals to use for this, recalling from Lemma~\ref{lem:interval-approx-lb} that there exists at least choice that will result in a good approximation.
\begin{proof}[Proof of Theorem~\ref{thm:tw-approx}]
Given any temporal graph $\cG$ with lifetime $T$, let us define $\widetilde{q}^{\,*}_c(\cG,d)$ to be
$$ \max_{\substack{0=t_0 < t_1 < \cdots < t_{\ell} = T \\ t_i - t_{i-1} \leq d \text{ for all } 1 \leq i \leq \ell}} \sum_{i=1}^{\ell} \widetilde{q}^{\,*}_c(\cG_{[t_{i-1}+1,t_i]}). $$
Thus, $\widetilde{q}^{\,*}_c(\cG,d)$ is the maximum possible sum of temporal $c$-modularities we can get from splitting the temporal graph into intervals of length at most $d$ and maximising temporal modularity independently on each interval. 
 By Lemmas \ref{lem:k_part_approx_bound}, \ref{lem:interval-approx-ub} and \ref{lem:interval-approx-lb}, we see that $\widetilde{q}^{\,*}_c(\cG,d)$ is a $(1-(\frac{1}{c}+\frac{2}{d}))$-approximation to $\widetilde{q}^{\,*}(\cG)$.  Given this approximation, we can easily compute a $(1-(\frac{1}{c}+\frac{2}{d}))$-approximation to $q^*(\cG)$ in time $O(mT)$, simply multiplying by the normalisation factor. It therefore suffices to show that we can compute $\widetilde{q}^{\,*}_c(\cG,d)$ in time $T^2n^{\mathcal{O}(cd)}c^{\mathcal{O}(\w d)}d^{\mathcal{O}(c)}$.
 
    We achieve this by dynamic programming.  Our goal is to compute a value $DP[s_1,s_2]$ for $1 \leq s_1 \leq s_2 \leq T$, such that 
    \begin{equation}
    \left(1-(\frac{1}{c}+\frac{2}{d})\right) \widetilde{q}_c^*(\cG_{[s_1,s_2]},d) \leq DP[s_1,s_2] \leq \widetilde{q}_c^*(\cG_{[s_1,s_2]},d).
    \label{eqn:DP-approx-bounds}
    \end{equation}
    We can then return the value of $DP[1,T]$ as our approximation to $\widetilde{q}_c^*(\cG,d)$.  We compute the table values as follows.
    We begin by initialising all table entries $DP[\cG,s_1,s_2]$ such that $s_2 - s_1 < d$ by running the algorithm of Theorem~\ref{thm:c-modularity} with input $\cG_{[s_1,s_2]}$ and recording the result.
    This follows from the fact that the intervals are smaller than $d$, so $\widetilde{q}^{\,*}_c(\cG_{[s_1,s_2]},d)=\widetilde{q}^{\,*}_c(\cG_{[s_1,s_2]})$.
    The remaining entries in the table, where $s_2 - s_1 \geq d$, are computed using the recurrence
    $$DP[s_1,s_2] = \max_{1 \leq i \leq d-1} \{ DP[s_1,s_1 + i] + DP[s_1+i+1,s_2]\}.$$

    \noindent
    \textbf{Correctness}: We argue, by induction on $s_2 - s_1$, that Equation \eqref{eqn:DP-approx-bounds} holds for all table entries.  For the base case, note that for entries with $s_2 - s_2 < d$ this follows immediately from the fact that the algorithm of Theorem $\ref{thm:c-modularity}$ gives an $ \left(1-(\frac{1}{c}+\frac{2}{d})\right)$-approximation to the non-normalised maximum temporal modularity.  Suppose now that $s_2 - s_1 \geq d$.  It follows immediately from the inductive hypothesis that $DP[s_1,s_2] \leq \widetilde{q}^{\,*}(\cG_{[s_1,s_2]},d)$, so it remains only to show the lower bound on $DP[s_1,s_2]$.  For this, fix a sequence $s_1-1=t_0 < t_1 < \dots < t_\ell = s_2$ such that $\widetilde{q}^{\,*}(\cG_{[s_1,s_2]},d) = \sum_{j=1}^\ell \widetilde{q}^{\,*}(\cG_{[t_{j-1}+1,t_j]})$ and $t_j - t_{j-1} \leq d$ for all $1 \leq j \leq \ell$.  Fix $i = t_1 - s_1 - 1$.  Then we can write 
    \begin{align*}
        \widetilde{q}^{\,*}(\cG_{[s_1,s_2]},d) &= \widetilde{q}^{\,*}(\cG_{[s_1,s_1+i]},d) + \sum_{j=1}^\ell \widetilde{q}^{\,*}(\cG_{[t_{j-1}+1,t_j]}) \\
        & \leq \widetilde{q}^{\,*}(\cG_{[s_1,s_1+i]},d) + \widetilde{q}^{\,*}(\cG_{[s_1+i+1,s_2]},d) \\
        & \leq \left(1-(\frac{1}{c}+\frac{2}{d})\right)^{-1} DP[s_1,s_1+i] + \left(1-(\frac{1}{c}+\frac{2}{d})\right)^{-1}DP[s_1+i+1,s_2],
    \end{align*}
    by the inductive hypothesis.  Rearranging gives the required lower bound.

    \vspace{6pt}
    \noindent
    \textbf{Running time:} We first consider the time required to compute the entries $DP[s_1,s_2]$ with $s_2-s_1<d$.  There are $\mathcal{O}(dT)$ such entries. For each we invoke the algorithm of Theorem \ref{thm:c-modularity}, which will run in time bounded by $n^{\mathcal{O}(cd)}c^{\mathcal{O}(\w d)}d^{\mathcal{O}(c)}$. Thus the time required for this initialisation phase is $\mathcal{O}(dT)\cdot n^{\mathcal{O}(cd)}c^{\mathcal{O}(\w d)}d^{\mathcal{O}(c)}$ which is bounded above by $Tn^{\mathcal{O}(cd)}c^{\mathcal{O}(\w d)}d^{\mathcal{O}(c)}$.  For each of the remaining entries, we take the minimum over $\mathcal{O}(d)$ sums of pairs of entries that have already been computed; computing each such entry takes time $\mathcal{O}(d)$.  Since there are in total $\mathcal{O}(T^2)$ entries in the table, this second phase requires time $\mathcal{O}(dT^2)$.  The total running time of the algorithm is at most $T^2n^{\mathcal{O}(cd)}c^{\mathcal{O}(\w d)}d^{\mathcal{O}(c)}$, as required.
    \qed
\end{proof}

\section{Conclusions and open problems}\label{sec:conclusion}

To our best knowledge, the algorithm described here is the first with performance guarantees for computing temporal modularity.  Given that temporal versions of problems that are tractable in the static case often become intractable even when the underlying graph is very severely restricted (e.g. even for paths \cite{mertziosMatchings23} or stars \cite{akridaStarExp21}), it is somewhat surprising that here we are able to generalise one approach from the static case, at the expense of a worse approximation factor.

A key open question is whether any of the exact parameterised algorithms for computing static modularity can be generalised to the temporal setting.  The vertex cover number of the underlying graph is a natural first candidate; note that this is strictly more restrictive than our treewidth-based approach, as vertex cover number bounds treewidth. 

The approach using vertex cover number in the static setting \cite{w1hard} is unlikely to be fruitful here without new ideas: it relies on partitioning vertices outside of a vertex cover into a small number of sets with \emph{identical} neighbourhoods; here, without also restricting the lifetime of the temporal graph, it is not possible to bound the number of different distinct neighbourhoods in terms of the vertex cover number.  A more fruitful approach might be to consider parameterisation by a temporal analogue of this parameter, the \emph{timed} vertex cover number \cite{casteigtsWaiting21}.

\begin{credits}
\subsubsection{\ackname}
 Jessica Enright and Kitty Meeks are supported by EPSRC grant EP/T004878/1. For the purpose of open access, the author(s) has applied a Creative Commons Attribution (CC BY) licence to any Author Accepted Manuscript version arising from this submission. 

\subsubsection{\discintname}
The authors have no competing interests to declare. 
\end{credits}

\newpage

{
%
%
%
\bibliographystyle{splncs04}
\bibliography{articles}
}

\end{document}